\documentclass[12pt]{amsart}
\usepackage{amsmath,amssymb,amsthm,amscd}

\def\e{\epsilon}

\def\a{{\alpha}}

\def\p{\partial}
\newcommand\R{{\mathbb R}}
\newcommand\T{{\mathbb T}}

\newcommand\C{{\mathbb C}}

\newtheorem{thm}{Theorem}[section]

\newtheorem{lem}{Lemma}[section]
\newtheorem{prop}{Proposition}[section]

\theoremstyle{definition}

\theoremstyle{remark}

\newtheorem{rem}{Remark}[section]

\begin{document}

\title[Lagrangian Mean Curvature flow for entire Lipschitz graphs II]
{Lagrangian Mean Curvature flow for entire Lipschitz graphs II}

\author{Albert CHAU}
\address{Department of Mathematics\\
University of British Columbia\\
Vancouver, B.C., V6T 1Z2\\
Canada}
\address{Department of Mathematics \\University of Washington\\Seattle, WA 98195\\U.S.A.}
\email{chau@math.ubc.ca}

\author{Jingyi CHEN}
\email{jychen@math.ubc.ca}

\author{Yu Yuan }
\email{yuan@math.washington.edu}
\thanks{2000 Mathematics Subject Classification.  Primary 53C44, 53A10.}
\thanks{The first two authors are partially supported by NSERC grants, and the third author is
partially supported by an NSF grant}
\date{\today}

\begin{abstract}
We prove longtime existence and estimates for solutions to a fully nonlinear Lagrangian parabolic equation with locally $C^{1,1}$ initial data $u_0$ satisfying either (1) $-(1+\eta) I_n\leq  D^2u_0 \leq (1+\eta)I_n$ for some positive dimensional constant $\eta$, (2) $u_0$ is weakly convex everywhere or (3) $u_0$ satisfies a large supercritical Lagrangian
phase condition. 
\end{abstract}

\maketitle

\section{introduction}
 When a family of smooth entire Lagrangian graphs in $\C^n$ evolve by the mean curvature flow their potentials $u:\R^n\times [0, T)\to \R$ will evolve, up to a time dependent constant, by the following fully nonlinear parabolic equation:
\begin{equation}\label{PMA}
\left\{%
\begin{array}{ll}
 &\dfrac{\partial u}{\partial t} =\displaystyle\sum_{i=1}^{n} \arctan \lambda_i\\
 &u(x, 0)=u_0(x)
\end{array}%
\right.
 \end{equation} 
 where $\lambda_i$'s are the eigenvalues of $D^2u$.  Conversely, if $u(x, t)$ solves \eqref{PMA}, then the graphs $(x, Du(x, t))$ in $\R^{2n}$ will evolve by the mean curvature flow up to tangential diffeomorphism.  The main result of the paper is the following.

\begin{thm}\label{entire}
There exists a small positive dimensional constant $\eta=\eta(n)$ such that if $u_0:\R^n \to \R$  is a $C^{1,1}$ function satisfying 
\begin{equation}\label{hesscond}
-(1+\eta) I_n\leq  D^2u_0 \leq (1+\eta)I_n
\end{equation}
  then
\eqref{PMA} has a unique longtime smooth solution $u(x, t)$ for all $t>0$ with initial condition $u_0$ such that the following estimates hold:

\begin{enumerate}
\item [(i)] $- \sqrt{3} I_n\leq D^2u\leq  \sqrt{3} I_n$  for all $t>0$.
\item [(ii)] $\sup_{x\in \R^n}|D^l u(x,t)|^2 \leq C_{l}/t^{l-2} $  for all $l\geq3$, $t>0$ and some $C_{l}$ depending only on $l$.
\item [(iii)] $Du(x, t)$ is uniformly H\"{o}lder continuous in time at $t=0$ with H\"{o}lder exponent $1/2$.
\end{enumerate}
\end{thm}

 In \cite{CCH1} Theorem \ref{entire} was proved for $\eta$ any negative constant in which case it was shown that \eqref{hesscond} is preserved for all $t>0$.  In particular, a priori estimates were established for any solution to \eqref{PMA} with $D^2u$ so bounded.  The estimates combined maximum principle arguments for tensors and a Bernstein theorem for entire special Lagrangians \cite{Y} via a blow up argument.  The estimates depended on the negativity of $\eta$ and could not be applied to the more general case of Theorem \ref{entire} even for $\eta=0$.  
 We overcome this through recent estimates in \cite{TY} for solutions to \eqref{PMA} satisfying certain Hessian conditions (cf. Theorem 2.1 which is Theorem 1.1 in \cite{TY}).  A particular case of Theorem \ref{entire} (similarly for Theorems \ref{entire3} and \ref{entire4}) is where $Du_0:\R^n \to \R^n$ is a lift of a map $f:\T^n \to \T^n$ and $\T^n$ is the standard $n$-dimensional flat torus.  In this ``periodic case", our estimates together with the results in \cite{Si} imply that the graphs $(x, Du(x, t))$  immediately become smooth after initial time and converge smoothly to a flat plane in $\R^{2n}$ (cf. \cite{CCH1,Si, Sm, SmMt}).  In the hypersurface case, the global and local behavior of mean curvature flow of Lipschitz continuous initial graphs has been studied in \cite{HE2, HE}.

After a coordinate rotation described in \S2 (see \eqref{rotatedcoord}), the condition $-I_n<  D^2u_0 < I_n$ corresponds to a convex potential in which case the right hand side of \eqref{PMA} is a concave operator.   This however is not the case under the weaker assumption  (i) in Theorem \ref{entire}.  This is interesting from a PDE standpoint as Krylov's theory for parabolic equations is for the concave operators.  

In light of the above, we apply Theorem \ref{entire} directly to the convex case in the following

 \begin{thm}\label{entire3}
Let $u_0:\R^n \to \R$ be a locally $C^{1,1}$ weakly convex function.  Then
\eqref{PMA} has a unique longtime smooth and weakly convex solution $u(x, t)$ with initial condition $u_0$ such that
\begin{enumerate}
\item [(i)] either $D^2u(x, t) >0$ for all $x$ and $t>0$ or there exists coordinates $x_1,...,x_n$ on $\R^n$ in which $u(x, t)=w(x_{k},...,x_n,t)$ on $\R^n\times[0, \infty)$ where $k>1$ and $w$ is convex with respect to $x_k,...,x_n$ for all $t>0$,\item [(ii)] $\sup_{x\in \R^n}|\nabla^l_t A(x,t)|^2 \leq C_{l}/t^{l+1} $  for all $l\geq0$, $t>0$ and some constant $C_{l}$ depending only on $l$ where $\nabla^l_t A(x,t)$ is the $l$th covariant derivative of the second fundamental form of the embedding $F_t:\R^n\to \R^{2n}$ given by $x\to (x, Du(x, t))$,
\item [(iii)] the Euclidean distance from each point of $F_t(\R^n)$ to $F_0(\R^n)$ in $\R^{2n}$ is H\"{o}lder continuous in time at $t=0$ with H\"{o}lder exponent $1/2$.
\end{enumerate}

\end{thm}
We also prove the following 

 \begin{thm}\label{entire4}
Let $u_0:\R^n \to \R$ be a locally $C^{1,1}$ function satisfying 
\begin{equation}\label{supercritical}
\displaystyle\sum_{i=1}^{n} \arctan \lambda_i  \geq (n-1)\frac{\pi}{2}.
\end{equation}
Then \eqref{PMA} has a unique longtime smooth solution $u(x, t)$ with initial condition $u_0$  such that \eqref{supercritical} is satisfied  with either strict inequality for all 
$t>0$ or equality for all $t\geq0$ in which case $u_0$ must be quadratic.  Moreover, $u(x, t)$  also satisfies (ii) and (iii) in Theorem \ref{entire3}.
\end{thm}
\begin{rem}\label{r1} Note that if $u_0$ satisfies \eqref{supercritical} then $u_0$ must be convex.    
\end{rem}

 As discussed above, after a coordinate rotation we may assume $D^2 u_0$ in Theorem \ref{entire3} satisfies the strict inequality $-I_n\leq  D^2u_0 < I_n$ in which case Theorem \ref{entire} immediately provides a longtime solution $u(x, t)$ to \eqref{PMA}.  In order for this to correspond to the desired longtime solution in the original coordinates we must first show $-I_n\leq  D^2u < I_n$ is preserved for all $t>0$ and this is the first main difficulty in proving (i) in Theorem \ref{entire3}.   This in particular will rule out the possibility of $\lambda_i (D^2u(x, t))=1$ for some $(x, t)$ which would correspond to a non-graphical (vertical) Lagrangian in the original coordinates.  The second main difficulty comes from showing that either $-I_n < D^2u$ for all $t>0$, or the solution splits off a quadratic term as in Lemma 4.2 and this will give (i) in Theorem \ref{entire3} after rotating back to the original coordinates.   

As for Theorem \ref{entire4}, by Remark \ref{r1}, if $u_0$ satisfies \eqref{supercritical} then it is automatically convex hence Theorem \ref{entire3} guarantees a longtime convex solution $u(x, t)$ to \eqref{PMA}.  The difficulty in showing \eqref{supercritical} is preserved for all $t>0$ comes from the fact that a maximum principle may not directly apply as $u_0$ is only $C^{1,1}$ with possibly unbounded Hessian.  Performing a similar but small $\sigma_0$ coordinate rotation, we can assume that 
 $-K(\sigma_0)I_n<  D^2u_0 < 1/K(\sigma_0) I_n$, for some constant $K(\sigma_0)$ which approaches zero as $\sigma_0\to 0$, and satisfies 
\begin{equation}\label{critical}\displaystyle\sum_{i=1}^{n} \arctan \lambda_i  \geq (n-1)\frac{\pi}{2}-n\sigma_0.\end{equation}
We then observe that the set of positive semi-definite real $n\times n$ matrices satisfying  \eqref{critical}  is a convex set $S$, and we approximate $u_0$ by convolution with the standard heat kernels, which has the effect of averaging elements in $S$, thus producing smooth approximations with bounded derivatives (of order $2$ and higher) and Hessians belonging to $S$.  We perform a further $\pi/4$ coordinate rotation after which the smooth approximated initial data satisfies (ii) in Theorem \ref{entire} and 
\begin{equation}\label{criticall}\displaystyle\sum_{i=1}^{n} \arctan \lambda_i  \geq (n-1)\frac{\pi}{2}-n\frac{\pi}{4}-n\sigma_0.\end{equation}
By Theorem \ref{entire} we then apply a maximum principle argument to show \eqref{criticall} is preserved starting from each approximate initial data.

 The outline of the rest of the paper is as follows.  In \S2 we provide preliminary results which will be used in the proofs of the theorems.  In particular, we state the a priori estimates in \cite{TY}.  Theorem \ref{entire} is proved in \S 3 and Theorems \ref{entire3} and \ref{entire4} are proved in  \S 4. 

\section{preliminaries}

In this section we establish some preliminary results.  

 \begin{prop}[ \cite{CCH1}; Proposition 5.1 ]\label{shorttime}
 Suppose $u_0:\R^n \to \R$ is a smooth function such that $\sup |D^l u_0| <\infty$ for each $l\geq 2$.  Then \eqref{PMA} has a smooth solution $u(x, t)$ on $\R^n \times [0, T)$ for some $T>0$ such that $\sup_{x\in \R^n}|D^l u(x, t)|<\infty$ for every $l\geq 2$ and $t\in [0, T)$.  
 \end{prop}
\begin{rem}
In Proposition 5.1 in \cite{CCH1} it was shown that the non-parametric mean curvature flow equation 
  \begin{equation}\label{GMCF}
\left\{%
\begin{array}{ll}
 & \dfrac{\partial f^{a}}{\partial t} = \displaystyle \sum_{i,j=1}^{n} g^{ij}(f)( f^{a})_{ij}\\
 &f(x, 0)=Du_0(x)
\end{array}%
\right.
\end{equation}
where $g^{ij}(f)$ is the matrix inverse of $g_{ij}(f):= \delta_{ij}+\sum_{a=1}^n f^a_i f^a_j$, has a short time solution $f(x, t)$ provided $u_0$ satisfies the conditions in Proposition \ref{shorttime}.   As explained in \cite{CCH1} (see Lemma 5.2), this in fact provides a short time solution $u(x, t)$ to \eqref{PMA} as in Proposition \ref{shorttime} such that $f(x, t)=Du(x, t)$ and the proof of Proposition 5.1 in \cite{CCH1} can also be adapted directly to \eqref{PMA} to establish Proposition \ref{shorttime}.  For convenience of the reader and completeness, we provide the details of this argument below.
\end{rem}  

\begin{proof} 

Let $C^{k+\alpha, k/2+\alpha/2}$ denote the standard parabolic H\"{o}lder spaces on $\R^n\times [0, 1)$.  Define

  $$ \mathcal{B}=\left\{v\in C^{2+\alpha,1+\frac\a2}| \ v(x,0)=0\right\}
  $$
and define a  map
$
F: \mathcal{B}\to C^{\alpha,\frac\a2}
$
by
$$
F(v)=\dfrac{\p v}{\p t}-\Theta(v)
$$
where $\Theta(v):=\sum_{i=1}^n \arctan \lambda_i(D^2(u_0+v))$.  Then the differential $DF_{v}$ at any $v\in \mathcal{B}$ is given by
$$
DF_{v}(\phi)=\dfrac{\p \phi}{\p t}- \displaystyle \sum_{i,j=1}^n g^{ij}(u_0+v)\phi_{ij}
$$
where $g^{ij}(u_0+v)$ is the matrix inverse of $I_n+ [D^2(u_0+v)]^2$.

 \begin{bf}Claim 1:\end{bf} $DF_{v}$ is a bijection from $T_{v} \mathcal{B}$ onto $T_{F(v)} C^{ 2+\alpha,1+\frac\a2}.$   

This follows from the general theory of linear parabolic equations on $\R^n\times [0, 1)$ with H\"{o}lder continuous coefficients.

 Now define functions $f_1,...,f_3$ on $\R^n$ recursively by 
\begin{equation}
\begin{split}
f_1&:=\Theta(u_0)\\
f_2&:=g^{ij}(u_0)\partial^2_{ij}f_1.\\
\end{split}
\end{equation}
 Then we see that $\sup_{\R^n}|D^l f_i|<\infty$ for every $i$ and $l$, and if we let $w_0=F(v_0)$ where $v_0=t f_1 + t^2  /2 f_2$, then a straightforward computation gives \begin{equation}\label{1}\partial^l_t F(v_0)(x, 0)=0\end{equation} for $l\leq1$ and \begin{equation}\label{2} \sup_{\R^n\times[0, 1)}|D_x^l D^m_t w_0|<\infty\end{equation} for every $l, m\geq 0$.  In particular $w_0 \in C^{\alpha,\frac\a2}$.  By the inverse function theorem there exists $\e>0$ such that $||w-w_0||_{\alpha, \frac\alpha2}<\e$ implies $F(v)=w$ for some $v\in\mathcal{B}$.

For any $0<\tau<1$, define $w_\tau$ by
\begin{equation}
w_\tau(x, t)=\left\{%
\begin{array}{ll}
 &   0, \hspace{2.1cm} t\leq \tau\\
 & w_0(x, t-\tau),\hspace{0.25cm} \tau<t<1.
\end{array}%
\right.
\end{equation}

\begin{bf}Claim 2:\end{bf} $||w_\tau-w_0||_{\alpha, \frac\alpha2}<\e$ for sufficiently small $\tau>0$.

By \eqref{1} and \eqref{2}, it follows that $w_{\tau}\in C^{ \alpha,\frac\a2}$ and $\|w_{\tau}\|_{\alpha,\frac\a2}$ is bounded uniformly and independently of $\tau$.  From this and the fact that $w_\tau-w_0$ converges uniformly to $0$ in $C^0$ as $\tau \to 0$, it is not hard to show the claim follows.

Hence by the inverse function theorem we have $F(v)=w_{\tau}$ for some $0<\tau<1$ and $v\in C^{2+\alpha, 1+\alpha/2}$.  In particular $u_0+v$ solves \eqref{PMA} on $\R^n\times[0, \tau]$.  Now the higher regularity of $u$ can be shown as follows.  For any $x_0\in \R^n$, consider the function $$\tilde{u}(x, t):= u(x  + x_0,t)-u( x_0, 0)-Du( x_0, 0)\cdot x.$$  Then $\tilde{u}(x, t)\in \mathcal{B}$ and still solves \eqref{PMA} on $\R^n\times[0, \tau]$.  Now we can write \eqref{PMA} as
\begin{equation}\label{bootstrap}
\begin{split}
 \dfrac{\p \tilde{u}}{\p t} &=  \sum_{i=1}^{n} \arctan \lambda_i (D^2\tilde{u})\\
 & = \int_0^1 \dfrac{\p}{\p s}  \left( \sum_{i=1}^{n}\arctan \lambda_i(D^2(s\tilde{u}))\right)ds\\
 &=\left(\int_0^1 g^{ij}(s\tilde{u})ds\right) \partial^2_{ij}\tilde{u}.
  \end{split}
  \end{equation}

Notice that $D\tilde{u}(0,0)=\tilde{u}(0,0)=0$ and that $D^2\tilde{u}(x,t)=D^2u(x+x_0,t)$ is uniformly bounded on $\R^n\times[0, \tau]$.  Now if we let $B(1)$ be the unit ball in $\R^n$ it follows from \eqref{PMA} that $\tilde{u}(x, t)$ and thus $D\tilde{u}(x,t)$ is uniformly bounded on $B(1)\times[0, \tau]$, giving $\tilde{u}(x, t)\in C^{2+ \alpha,1+\frac{1}{2}+ \frac\a2}(B(1)\times[0, \tau])$.  In particular, by freezing the symbol $a^{ij}:=\int_0^1 g^{ij}(s\tilde{u})ds$ in \eqref{bootstrap}, we can view \eqref{bootstrap} as a linear  parabolic equation for $\tilde{u}$ with coefficients uniformly bounded in $C^{\alpha,\frac{1}{2}+ \frac\a2}(B(1)\times[0, \tau])$.   Now applying the local parabolic Schauder estimates (Theorem 8.12.1, \cite{Kr}) and a standard bootstrapping argument to \eqref{bootstrap} we may then bound the $C^{l+ \alpha}$ norm of $\tilde{u}(x, t)$ on $B(1)$ by a constant depending only on $t$ and $l$.   

 Now the fact that $v$ is smooth with bounded derivatives as in the theorem follows by repeating the above argument for any $x_0 \in \R^n$.
\end{proof}

  \begin{lem}[\cite{CCH1}; Lemma 5.1] \label{approx}
Let $u_0:\R^n \to \R$ be a $C^{1,1}$ function satisfying  $-C_0 I_n \leq D^2 u_0 \leq C_0 I_n$ for some constant $C_0 >0$. Then there exists a sequence of smooth functions $u^k_0:\R^n\to \R$  such that
 \begin{enumerate}
\item [(i)] $u^k_0 \to u_0$ in $C^{1+\alpha}(B_R(0))$ for any $R$ and $0<\alpha<1 $,
\item [(ii)] $-C_0 I_n \leq D^2 u^k_0 \leq C_0 I_n$ for every $k$, 
 \item [(iii)] $\sup_{x\in\R^n}|D^l u^k_0| <\infty$ for every $l \geq 2$ and $k$.
 \end{enumerate}
 \end{lem}
 
 \begin{proof}
 Let 
\begin{equation}\label{approxx}
u^k_0(x)=\int_{\R^n} u_0(y)K\left(x, y, \frac{1}{k}\right)dy
\end{equation}
where $K(x, y, t)$ is the standard heat kernel on $\R^n\times (0, \infty)$.    Conditions (i) and smoothness of $u^k_0$ are easily verified.  By assumption, $D_y^2 u_0(y)$ is a well defined and uniformly bounded function almost everywhere on $\R^n$ and we may write 
$$
D_x^lu^k_0 (x)=\int_{\R^n} D_y^2 u_0(y)D_x^{l-2}K\left(x, y, \frac{1}{k}\right)dy
$$ 
for every $l\geq 2$ from which it is easy to see that conditions (ii) and (iii) is also true.  \end{proof}

\begin{thm}[\cite{TY}; Theorem 1.1 \label{estimateTY}]
Let $u(x, t)$ be a smooth solution to \eqref{PMA} in $Q_1 \subset \R^n\times (-\infty, 0]$.  When $n\geq4$ we also assume that at least one of the following conditions holds in $Q_1$
\begin{enumerate}
\item[(i)] $\displaystyle\sum_{i=1}^{n} \arctan \lambda_i \geq (n-2)\frac{\pi}{2},$
\item[(ii)] $3+\lambda_i^2+2\lambda_i\lambda_j\geq 0$ for  all $1\leq i, j \leq n$.
\end{enumerate}
Then we have
\begin{equation}\label{TYe1} [u_t]_{1, \frac{1}{2}; Q_{1/2}} +[D^2 u]_{1, \frac{1}{2}; Q_{1/2}} \leq C(\| D^2u\|_{L^{\infty}(Q_1)}).
\end{equation}
\end{thm} 
Here $Q_r(x, t)=B_r(x)\times[t-r^2, t] \subset \R^n\times(-\infty, 0]$, and $Q_r:=Q_r(0, 0)$.  We refer to \cite{TY} for further notations and definitions used in Theorem \ref{estimateTY}.  

 \begin{lem}\label{preserve1}
Suppose $u_0:\R^n \to \R$ is a $C^{1,1}$ function satisfying 
\begin{equation}\label{hesscond2}
- I_n\leq D^2u_0 \leq I_n
\end{equation}
and that $u(x, t)\in C^{\infty}(\R^{n}\times(0, T))\bigcap C^{0}(\R^{n}\times[0, T))$ is a solution to \eqref{PMA} and satisfies $u(x, 0)=u_0$.  Then \eqref{hesscond2} is preserved for all $t$. 
\end{lem}

\begin{proof}
We begin by establishing the following special case

{\bf Claim:} If $u(x, t)$ is a smooth solution of \eqref{PMA} on $\R^{n}\times[0, T)$ satisfying \begin{enumerate}
\item [(i)] $\sup_{\R^n}|D^lu(x, t)|<\infty$ for every $t\in[0, T)$ and $l\geq 2,$
\item [(ii)] $u(x, 0)$ satisfies $\left(  -1+\delta\right)  I_n \leq D^{2}u\left( x,0\right)  \leq\left(1-\delta\right) I_n$ for some $\delta >0,$
\end{enumerate}
 then $u(x, t)$ satisfies $\left(-1+\delta\right)  I_n \leq D^{2}u\left( x,t\right)  \leq\left(1-\delta\right) I_n$ for each $t\in (0, T)$.

This was established in Lemma 4.1 in \cite{CCH1} and we provide a different proof of this here.  We begin by describing a change of coordinate which we will use at various places throughout the paper.  Let $z^j=x^j+\sqrt{-1}y^j$ and $w^j=r^j+\sqrt{-1}s^j$ ($j=1,...,n$) be two holomorphic coordinates on $\C^n$ related by \begin{equation}\label{rotatedcoord}z^j =e^{\sqrt{-1}\sigma} w^j\end{equation} for some constant $\sigma$.  Then as described in \cite{Y2} (see p.1356), if $L=\{(x, u_0(x)) | x\in \R^n\}$ in $\C^n$ is represented as $L=\{(r, v_0(r)) | r\in \R^n\}$ in the coordinates $w^j$, then $v_0$ satisfies 
\begin{equation}\label{rotatedlambda's} \arctan \lambda_i(D^2v_0) = \arctan \lambda_i(D^2u_0)-\sigma.\end{equation}
  Now by (ii) in the claim, as described in \cite{Y} we may choose $\sigma=-\pi/4$ and obtain such a new graphical representation of $L$ and the new potential function will satisfy
\begin{equation}\label{rotatedhesscond3}\frac{\delta}{2-\delta}I_n \leq D^{2}v_0  \leq\frac{2-\delta}{\delta}I_n. \end{equation} The claim will be established once we show \eqref{rotatedhesscond3}
is preserved for any $\delta>0$.  Differentiating \eqref{PMA} twice with respect to any coordinate direction $x_k$ yields 
\begin{equation}\label{preserved1}\displaystyle \sum_{i,j=1}^{n} g^{ij}\partial_{ij}v_{kk}-\partial_{t}v_{kk}= \sum_{l,m=1}^{n} g^{ll}g^{mm}\left(  \lambda_{l}+\lambda_{m}\right)  v_{lmk}^{2} \geq 0\end{equation} 
where the subscripts of $v$ denote partial differentiation.  Now fix any vector $V\in \R^n$ and any point $(r_0, t_0)$ note that $V^T D^2v(r_0, t_0) V=v_{VV}(r, t)$ where $v_{VV}(r, t)$ is just the second derivative of $v(r_0, t_0)$ in the direction $V$.  It follows from \eqref{preserved1} that the function $$f(r, t)=V^{T}\left(D^{2}v\left(r,t\right)  -\frac{2-\delta}{\delta}I\right)V$$ satisfies $$\left( \displaystyle \sum_{i,j=1}^{n} g^{ij}\partial^2_{ij}-\partial_{t}\right)f(r, t) \geq 0$$ at any $(r, t)$ in $\R^{n}\times[0, T)$.  Now note that by our assumption on the derivatives of $u$ we have that $g_{ij}(r, t)$ is uniformly equivalent to the Euclidean metric on $\R^n$ uniformly for $t \in [0,T_1]$ with $T_1 < T$, while $g_{ij}(r, t)$ and $f(r, t)$ are also continuous on $\R^{n}\times[0, T)$ the maximum principle (Theorem 9, p.43, \cite{F}) then implies $f(r, t)\leq 0$ for all $t$.  We can similarly prove that $f(r, t)\geq 0$ for all $t$.  This establishes  the claim.

 Now let $u_0$ and $u(x, t)$ be as in the lemma, and let $u_0^k$ be a sequence as in Lemma \ref{approx}.  Fix some sequence $\delta_k \to 0$ and consider the sequence $v_0^k = (1-\delta_k) u_0^k$.  Then by Proposition \ref{shorttime} there exists a positive sequence $T_k$ such that for each $k$ there is a smooth solution $v_k(x, t)$  of \eqref{PMA} on $\R^{n}\times[0, T_k)$ with initial condition $v_0^k$ and $\sup_{x\in\R^n}|D^l v^k(x, t)| <\infty$ for every $l \geq 2$ and $t\in[0, T_k)$.  For each $k$, assume that $T_k$ is the maximal time on which the solution $v_k$ exists.   By the above claim we also have $\left( -1+\delta_k\right)  I_n \leq D^{2}v_k\left( x,t\right)  \leq\left(1-\delta_k\right)I_n $ for each $t\in [0, T_k)$ and the main theorem in \cite{CCH1} then implies $\sup_{x\in \R^n}|D^l v_k(x,t)|^2 \leq C_{l, k}/t^{l-2} $ for all $l\geq3$, and some constant $C_{l, k}$ depending only on $l$ and $\delta_k$ and it follows that $T_k =\infty$.   In fact, the local estimates in Theorem \ref{estimateTY} can be used to remove the dependence on $\delta_k$ in these bounds.  Indeed, fix some $k$, $T\in(0,\infty)$ and $x'\in \R^n$ and let
\begin{equation}\label{rescaling}w_k(y, s)=\frac{1}{T}\left(v_k(y\sqrt{T}+x', sT+T)-v_k( x', T)-Dv_k( x', T)\cdot y\right).\end{equation}Then  we have $w_k(0, 0)=Dw_k(0,0)=0$, and $w_k(y, s)$ solves \eqref{PMA} on $\R^n\times[-1,0]$ and satisfies  $-(1-\delta_k)I_n \leq D^2 w_k \leq (1-\delta_k)I_n$  for all $(y, s)\in\R^n\times[-1,0]$. Applying Theorem \ref{estimateTY}  then gives
 \begin{equation}\label{estimate1}\sup_{(x,t)\in B_{\sqrt{T}/2}(x')\times[(3T/4), T]}\left|D^3 v_k(x, t)\right|^2=\sup_{(y,s)\in Q_{1/2}}\dfrac{1}{T}\left|D^ 3 w_k(y, s)\right|^2 \leq \dfrac{C}{T}
\end{equation}
where $B_{\sqrt{T}/2}(x')$ is the ball of radius $\sqrt{T}/2$ centered at $x'\in \R^n$ and $C$ is some constant independent of $k$.  Noting that $x'\in\R^n$ and $T\in(0,\infty)$ were arbitrary we obtain
 \begin{equation}\label{estimate2}\sup_{x\in \R^n}\left|D^3 v_k(x, t)\right|^2 \leq \dfrac{C}{t}
\end{equation}
for all $t\in(0, \infty)$  and it follows from a scaling argument, described in the proof of Lemma 5.2 in \cite{CCH1}, that for every $t\in(0, \infty)$ and $l\geq3$ we may have 
 \begin{equation}\label{estimate3}\sup_{x\in \R^n}\left|D^l v_k(x, t)\right|^2 \leq \dfrac{C_l}{t^{l-2}}
\end{equation}
for some constant $C_l$ depending only on $l$.
 
  From \eqref{estimate3} we conclude that the $v_k(x, t)$'s have a subsequence converging to a function $v(x, t)$ on $\R^n\times[0,  \infty)$ where the convergence is smooth on compact subsets of $\R^n\times(0,  \infty)$.  In particular,  by construction we have that $v(x, t)$ is smooth and solves \eqref{PMA} on $\R^n\times(0,  \infty)$, satisfies \eqref{hesscond2} for every $t\in [0, \infty)$ and $v(x, 0)=u_0(x)$.  Moreover, by \eqref{PMA} we have $|\partial_t v (x, t)|\leq \frac{n\pi}{2}$ for all $(x, t) \in\R^n\times[0,  \infty)$ from which we conclude that $v\in C^{0}(\R^{n}\times[0, T))$.
 
  It now follows by the uniqueness result in \cite{chen-pang} that $u(x, t)=v(x, t)$ for all $t\in [0, T)$, and thus $u(x, t)$ also satisfies \eqref{hesscond2} for every $t\in [0, T)$.  This completes the proof of the lemma.
 \end{proof}

 We now apply the above results to prove the following proposition. 

\begin{prop}\label{preserve2}
There exists a dimensional constant $\eta=\eta(n)>0$ such that for every $T>0$ the following holds: if $u(x, t)$ is a smooth solution to \eqref{PMA} on $\R^n\times[0, T)$ such that  $-(1+\eta)I \leq D^2u \leq (1+\eta)I$ at $t=0$ and $\sup_{x\in\R^n}|D^l u| <\infty$ for each $t\in [0, T)$ and $l\geq 2$, then $u(x, t)$ satisfies  $-\sqrt{3}I_n \leq D^2u \leq \sqrt{3}I_n$ for all $t\in[0, T)$.
\end{prop}

\begin{proof} Suppose otherwise.  Then there exists a sequence $\eta_k \to 0$ and a sequence of smooth $u_k(x, t)$ each solving \eqref{PMA} on $\R^n \times[0, T_k)$ where $T_k >0$, and each satisfying 
\begin{enumerate}
\item [(a)] $-(1+\eta_k)I_n \leq D^2u_k \leq (1+\eta_k)I_n$ at $t=0$,
\item [(b)] $\sup_{x\in\R^n}\left|D^l u_k \right| <\infty$ for each $t\in [0, T_k)$ and $l\geq 2$,
\item [(c)] $\left|\lambda_i(D^2 u_k(x_k, t_k))\right| > \displaystyle \sqrt{3}$  for some $(x_k, t_k) \in \R^n\times[0, T_k)$ and some $i$.
\end{enumerate}
Then by (a) and (b)  it is not hard to show that there exists a sequence  $R_k$ with $R^2_k \in (0, T_k)$ satisfying
\begin{enumerate}
\item [(A)] $-\sqrt{3}I_n \leq D^2u_k \leq \sqrt{3}I_n$ for all $t\in [0, R_k^2)$,
\item  [(B)] $\left|\lambda_i(D^2 u_k(x_k, t_k))\right| = \sqrt{3/2}$  for some $(x_k, t_k) \in \R^n\times[0, R_k^2)$ and some $i$.
\end{enumerate}
Now consider the sequence $$v_k(x, t):=\frac{1}{t_k^2} \left(u_k(x t_k + x_k, t_k^2 t + t_k^2)-u_k( x_k, t_k)-Du_k(  x_k, t_k^2)\cdot x\right)$$ each solving \eqref{PMA}  on $\R^n\times[-1, 0]$ and each satisfying
\begin{enumerate}
\item [(i)] $-(1+\eta_k)I_n \leq D^2v_k \leq (1+\eta_k)I_n$ at $t=0$ and $-\sqrt{3}I_n \leq D^2v_k \leq \sqrt{3}I_n$ $\forall t>0$,
\item [(ii)] $\displaystyle \left|\lambda_i (D^2 v_k)(0, 0)\right|=  \sqrt{3/2}$ for some $i$,
\item [(iii)] $v_k(0, 0)=Dv_k(0, 0)=0$.
\end{enumerate}
Then as assumption (ii) in Theorem \ref{estimateTY} is satisfied, we may apply the estimates there as in the proof of Lemma \ref{preserve1} to show that the $v_k(x, t)'s$ have a subsequence converging to a function $v(x, t)\in C^{\infty}(\R^n\times(-1, 0)) \bigcap  C^{0}(\R^n\times[-1, 0])$ such that in addition we have $\sup_{x\in\R^n} |D^3 v(x, t)|$ bounded independent of $t\in[-1/2, 0]$.  Moreover, by construction $v(x, t)$ solves \eqref{PMA} and satisfies $-I_n \leq D^2v(x, -1) \leq I_n$ in the $L_{\infty}$ sense and $|\lambda_i(D^2v(0,0))| =\sqrt{3/2}$ for some $i$.  Together, these facts contradict Lemma \ref{preserve1}.  
\end{proof}
\begin{rem} Noting that \eqref{TYe1} in Theorem \ref{estimateTY} holds in general when $n\leq3$, we observe that when $n\leq3$ we can replace $\sqrt{3}$ in Proposition \ref{preserve2} with any positive constant $C>0$.
\end{rem}

\section{Proof of Theorem \ref{entire}}

\begin{proof}[Proof of Theorem \ref{entire}]
 Let $u_0$ be as in Theorem \ref{entire} where $\eta>0$ is as in Proposition \ref{preserve2}.  Let $u^k_0$ be a sequence of approximations as in Lemma \ref{approx}.  By Proposition \ref{shorttime} we have smooth short time solutions $u_k(x, t)$ to \eqref{PMA} with initial condition $u_k(x, t)=u_0^k(x)$.  Moreover, by Proposition \ref{preserve2} we have $-\sqrt{3}I_n \leq D^2u \leq \sqrt{3}I_n$  for all $(x, t)$.  We will let $\R^n\times[0, T_k)$ be the maximal space time domain on which $u_k(x, t)$ is defined.  
  
   Then by a rescaling argument and applying Theorem \ref{estimateTY} as in the the proof of Lemma \ref{preserve1}, we can show that for each $k$, $T_k=\infty$ and $u_k(x, t)$ satisfies the estimates in \eqref{estimate3} for all $l\geq 3$ and $t>0$.  In particular, we argue as in the last two paragraphs of the proof of Lemma \ref{preserve1} that some subsequence of the $u_k(x, t)$'s converge to a function $u(x, t)$ solving \eqref{PMA} on $\R^n\times[0, \infty)$ satisfying (i) and (ii) in the conclusions of Theorem \ref{entire}. 
   
   We now show that $Du(x, t)$ satisfies conclusion (iii) in Theorem \ref{entire}.
By differentiating \eqref{PMA} once in space and using (i) and the estimates in (ii) for $l=3$ we may estimate as follows for any $x\in \R^n$ and $t>t' >0$:

\begin{equation}
\begin{split}
 \frac{\left|Du(x, t) -Du(x, t')\right|}{(t-t')^{1/2}}\leq &\frac{\left|\displaystyle\int_{t'}^{t} \partial_s Du(x, s) ds\right|}{(t-t')^{1/2}}\\
\leq &\frac{C\displaystyle\int_{t'}^{t}\left|D^3u(x, s)\right| ds }{(t-t')^{1/2}}\\
\leq &\frac{C\displaystyle\int_{t'}^{t}s^{-1/2} ds }{(t-t')^{1/2}}\\
\leq &C
\end{split}
\end{equation}
for some constant $C$ independent of $x$, $t$ and $t'$.  The uniqueness of $u(x,t)$ follows from the uniqueness result in \cite{chen-pang}. \end{proof}

 \section{Proofs of Theorem \ref{entire3} and Theorem \ref{entire4}}

We begin by establishing the following lemmas
\begin{lem}\label{l1}
Let $v(r, t)$ be a solution to \eqref{PMA} as in Theorem \ref{entire} and assume $-I_n\leq D^2 v(r, t)\leq I_n$ for all $r$ and $t$. Then if $\lambda_1(r', t')=1$ at some point where $t'>0$, then $\lambda_1(r, t)=1$ for all $(r,t) \in \R^n \times[0, t')$. Similarly if $\lambda_1(r', t')=-1$ at some point where $t'>0$, then $\lambda_1(r, t)=-1$ for all $(r,t) \in \R^n \times[0, t')$.
 \end{lem}

\begin{proof}
  In \cite{YW}, the authors consider a solution $v$ to the elliptic equation corresponding to \eqref{PMA}:
\begin{equation}\label{EMA}
\displaystyle\sum_{i=1}^{n} \arctan \lambda_i = C\\
 \end{equation} 
where $C$ is some constant. By twice differentiating  \eqref{EMA} and the characteristic equation $\det (D^2v +\lambda_i I_n )=0$  they obtained a formula for $\sum_{a,b=1}^n g^{ab}\partial^2_{ab} \ln \sqrt{1+\lambda_i^2}$ at any point where $\lambda_i$ is a non-repeated eigenvalue for $D^2v$ (\cite{YW}; Lemma 2.1).  By essentially the same calculations we may differentiate the parabolic equation \eqref{PMA} and the characteristic equation $\det (D^2v +\lambda_i I_n )=0$ twice in space, and also differentiate the characteristic equation once in time, to obtain the exact same formula for $(\sum_{a,b=1}^n g^{ab}\partial^2_{ab}-\partial_t )\ln \sqrt{1+\lambda_i^2}$ at any point where $\lambda_i$ is a non-repeated eigenvalue for $D^2v$.  Namely, if $\lambda_i$ is a non-repeated eigenvalue of $D^2v$ at a point $(r_0, t_0)$ then the following holds at $(r_0, t_0)$ (after making a linear change of coordinates on $\R^n$ so that $D^2v(r_0, t_0)$ is diagonal): 

\begin{equation}\label{lambda_n}
\displaystyle
\begin{split} 
&\left(\sum_{a,b=1}^n g^{ab}\partial^2_{ab}-\partial_t \right)\ln \sqrt{1+\lambda_i^2}\\ 
&  =\left(  1+\lambda_{i}^{2}\right)  h_{iii}^{2}\\
&  +\sum_{\alpha\neq i}\frac{2\lambda_{i}}{\lambda_{i}-\lambda_{\alpha}}\left(  1+\lambda_{i}\lambda_{\alpha}\right)  h_{\alpha\alpha i}^{2}+\sum_{\alpha\neq i}\left[  1+\lambda_{i}^{2}+\frac{2\lambda_{i}}{\lambda_{i}-\lambda_{\alpha}}\left( 1+\lambda_{i}\lambda_{\alpha}\right)  \right]h_{ii\alpha}^{2}\\
&  +\sum_{\substack{\alpha<\beta\,\\\alpha,\beta\neq i}}2\lambda_{i}\left(\frac{1+\lambda_{i}\lambda_{\alpha}}{\lambda_{i}-\lambda_{\alpha}}+\frac{1+\lambda_{i}\lambda_{\beta}}{\lambda_{i}-\lambda_{\beta}}\right)h_{\alpha\beta i}^{2}
\end{split}
\end{equation}
where $h_{\alpha \beta \gamma}(r, t)$ is the second fundamental form of the embedding $F(r, t)=(r, Dv(r, t))$ of $\R^n$ to $\R^{2n}$.

 {\bf Claim:}  If $\lambda_1(r', t')=1$ at some point where $t'>0$, then $\lambda_1(r, t)=1$ for all $(r,t) \in \R^n \times(0, t']$.  

We will always assume that  $1\geq \lambda_1\geq \lambda_{2}\geq \cdot\cdot\cdot \geq \lambda_n \geq-1$ where the upper and lower bounds are given by Lemma \ref{preserve1}.   Now suppose that $1$ is an eigenvalue of multiplicity $k$ and consider the function  $f=\sum_{i=1}^k \ln \sqrt{1+\lambda_{ i}^2}$.  Then $f$ is a smooth function in a space-time neighborhood $U\times (t'-\epsilon, t'+\epsilon)$ of $(r',t')$ (see \cite{S}) and attains a maximum value in $U\times (t'-\epsilon, t'+\epsilon)$ at $(r',t')$.  Now we want to compute the evolution of $f$ in $U\times (t'-\epsilon, t'+\epsilon)$.  We illustrate how to do this first at some point where $\lambda_1,...,\lambda_k$ are all distinct.  In this case we may apply
 \eqref{lambda_n} separately to each term in $f$, and after some computation we obtain
\begin{equation}\label{111}
\begin{split}
\left(\sum_{a,b=1}^n  g^{ab}\partial^2_{ab}-\partial_t \right)& \sum_{i=1}^k \ln \sqrt{1+\lambda_{ i}^2}\\
&=\sum_{\gamma\leq k}\left(1+\lambda_{\gamma}^{2}\right)  h_{\gamma\gamma\gamma}^{2}+I+I\!I\\
&\geq I+I\!I
\end{split}
\end{equation}
where 
\begin{equation}\label{222}
\begin{split}
& 
\begin{array}
[c]{l}
\displaystyle
I=\sum_{\alpha<\gamma\leq k}\left(  3+\lambda_{\alpha}^{2}+2\lambda_{\alpha}\lambda_{\gamma}\right)  h_{\alpha\alpha\gamma}^{2}+\sum_{\alpha\leq k<\gamma}\frac{3\lambda_{\alpha}-\lambda_{\gamma}+\lambda_{\alpha}%
^{2}\left(  \lambda_{\alpha}+\lambda_{\gamma}\right)  }{\lambda_{\alpha}-\lambda_{\gamma}}h_{\alpha\alpha\gamma}^{2}\\
\displaystyle
\,\,\,\,\,\,+\sum_{\alpha<\gamma\leq k}\left(  3+\lambda_{\gamma}^{2}+2\lambda_{\gamma}\lambda_{\alpha}\right)  h_{\gamma\gamma\alpha}^{2}+\sum_{\alpha\leq k<\gamma}\frac{2\lambda_{\alpha}\left(  1+\lambda_{\alpha}\lambda_{\gamma}\right)  }{\lambda_{\alpha}-\lambda_{\gamma}}h_{\gamma \gamma\alpha}^{2}\vspace{20pt}
\end{array}
 \\
&
\begin{array}
[c]{l}
\displaystyle
I\!I=2\sum_{\alpha<\beta<\gamma\leq k}\left(  3+\lambda_{\alpha}\lambda_{\beta}+\lambda_{\beta}\lambda_{\gamma}+\lambda_{\gamma}\lambda_{\alpha}\right)  h_{\alpha\beta\gamma}^{2}\\
\displaystyle
\,\,\,\,\,\,\,\,+ \,2\sum_{\alpha<\beta\leq k<\gamma}\left[  1+\lambda_{\alpha}\lambda_{\beta}+\lambda_{\beta}\left(  \frac{1+\lambda_{\beta}\lambda_{\gamma} 
}{\lambda_{\beta}-\lambda_{\gamma}}\right)  +\lambda_{\alpha}\left(\frac{1+\lambda_{\alpha}\lambda_{\gamma}}{\lambda_{\alpha}%
-\lambda_{\gamma}}\right)  \right]  h_{\alpha\beta\gamma}^{2}\\
\displaystyle
\,\,\,\,\,\,\,\,+ \,2\sum_{\alpha\leq k<\beta<\gamma}\lambda_{\alpha}\left(  \frac{1+\lambda_{\alpha}\lambda_{\beta}}{\lambda_{\alpha}-\lambda_{\beta}
}+\frac{1+\lambda_{\alpha}\lambda_{\gamma}}{\lambda_{\alpha}
-\lambda_{\gamma}}\right)  h_{\alpha\beta\gamma}^{2}
\end{array}
  \\
\end{split}
\end{equation}
 where $I$ corresponds to summing the second and third term on the right hand side of \eqref{lambda_n} for $i=1,...,k$ and $I\!I$ corresponds to summing the fourth term on the right hand side of \eqref{lambda_n} for $i=1,...,k$.  Our derivation above only applies at a point where $\lambda_1,...,\lambda_k$ are all distinct, and thus cannot be used directly to calculate the evolution of $f=\sum_{i=1}^k \ln \sqrt{1+\lambda_{ i}^2}$ at $(r', t')$.  We now remove this assumption on the distinctness of eigenvalues by the approximation argument below.

  Consider the function $$v_{m}(r, t):=v(r, t)-\frac{1}{m}\sum_{j=1}^k j r_j^2.$$  Then for sufficiently large $m$, in some space-time neighborhood of $(r', t')$ which we still denote as $U\times (t'-\epsilon, t'+\epsilon)$ the eigenvalues $\lambda_{i, m}$ of $D^2 v_m$ will be between $-1$ and $1$ while the $k$ largest eigenvalues will be non-repeated.  Thus the function  $\ln \sqrt{1+\lambda_{i,m}^2}$ is smooth in $U\times (t'-\epsilon, t'+\epsilon)$ for each $i$.   On the other hand, by \eqref{PMA} and the definition $v_{m}$ we have
   \begin{equation}\label{PMAapprox}
\dfrac{\partial v_m}{\partial t} =\displaystyle\sum_{i=1}^{n} \arctan \lambda_{i,m}+w_m
 \end{equation} 
 where $$w_{m}=\displaystyle\sum_{i=1}^{n} \arctan \lambda_i(v)-\displaystyle\sum_{i=1}^{n} \arctan \lambda_i(v_m).$$ Note that $w_m$ approaches zero smoothly and uniformly on compact subsets of $U\times (t'-\epsilon, t'+\epsilon)$ as $m \to \infty$.
 By \eqref{PMAapprox},  the above referenced derivation of \eqref{lambda_n} and by \eqref{222} we have
 
\begin{equation}\label{333}
\begin{split}
\left(\sum_{a,b=1}^n  g_m^{ab}\partial^2_{ab}-\partial_t \right)& \sum_{i=1}^k \ln \sqrt{1+\lambda_{i,m}^2}\\
&=\sum_{\gamma\leq k}\left(1+\lambda_{\gamma,m}^{2}\right)  h_{\gamma\gamma\gamma}^{2}+I_m+I\!I_m\\
&+\sum_{i=1}^{k}\frac{\lambda_{i,m}}{1+\lambda_{i,m}^{2}}\left(
w_{m}\right)  _{ii}\\
& \geq\sum_{i=1}^{k}\frac{\lambda_{i,m}}{1+\lambda_{i,m}^{2}}\left(w_{m}\right)  _{ii}
\end{split}
\end{equation}
in $U\times (t'-\epsilon, t'+\epsilon)$ where $I_m$ is obtained by replacing $\lambda_{\alpha}$ and  $\lambda_{\beta}$ in $I$ by $\lambda_{\alpha,m}$ and $\lambda_{\beta,m}$ respectively, and $I\!I_m$ is obtained similarly. We have also used the fact that $I_m, I\!I_m$ is nonnegative.  Letting $m  \to \infty$, we conclude that $(\sum_{a,b=1}^n  g^{ab}\partial^2_{ab}-\partial_t ) f \geq 0$ on $U\times (t'-\epsilon, t'+\epsilon)$ and thus $f=k \ln \sqrt{2}$ in $U\times (t'-\epsilon, t']$ by the strong maximum principle (Theorem 1, p.34, \cite{F}).  

 Now for any $(r'', t'')\in\R^n\times[0, t']$ let $\gamma(s):[0, 1]$ be a line segment in space-time such that $\gamma(0)=(r',t')$ and $\gamma(1)=(r'',t'')$.  Let $A$ be the set of $\bar{s} \in[0, 1]$ for which $\lambda_1(\gamma(s))=1$ for all $s\in[0, \bar{s}]$.  Then the above argument shows that $A$ is in fact open and non-empty.  Moreover, $A$ is clearly closed by continuity and we then conclude that $A=[0, 1]$ and in particular,  $\lambda_1 (r'',t'')=1$.  This established the claim and thus the first statement in the conclusion of the lemma.

 By considering the solution $-v(r, t)$ to \eqref{PMA}, we likewise conclude the second statement in the concslusion of the lemma is true.\end{proof}

\begin{lem}\label{l2}
Let $v(r, t)$ be a solution to \eqref{PMA} as in Theorem \ref{entire} and assume that $-I_n<D^2 v(r, t)\leq I_n$ for all $r\in\R^n$ and $t\in[0, \infty)$. Then either $D^2 v(r, t)< I_n$ for all $r$ and $t>0$ or there exist coordinates $r_1,...,r_n$ on $\R^n$ in which we have $v(r, t)=\frac{r_1^2}{2}+\cdot\cdot\cdot+ \frac{r_k^2}{2}+w(r_{k+1},...,r_n,t)$ on $\R^n\times[0, \infty)$ where $-I_n<D^2 w(r, t)< I_n$ for all $r$, $t> 0$ and $k>1$.
 \end{lem}
\begin{proof}
We begin by establishing the following claims.

{\bf Claim 1:}  If $v_{11}(r', t')=1$ at some point $(r', t')$ with $t'>0$, then $v_{11}=1$ on $\R^n\times[0, t']$.  

 By a rotation of coordinates on $\R^n$, we may assume that $D^2 v(r',t')$ is diagonal.  Since $D^2v>-I_n$ there exists some space time neighborhood  $U_1\times (t'-\epsilon, t'+\epsilon)$ of $(r', t')$ in which $-(1-\delta)I_n \leq D^2 v(r, t)\leq I_n$ for some $\epsilon, \delta>0$.  By \eqref{rotatedlambda's} it follows that for some choice of $\sigma\in (0, \pi/4)$, we may change coordinates on $\C^n$ (from $w^j$ to $z^j$) using \eqref{rotatedcoord} so that the local family of Lagrangian graphs $L=\{(r, D v(r, t)) | (r, t)\in U_1\times (t'-\epsilon, t'+\epsilon)\}$ is represented in the new coordinates as $L=\{(x, D u(x, t)) | (x, t)\in U_2\times (t'-\epsilon, t'+\epsilon)\}$ for some space time neighborhood $U_2\times (t'-\epsilon, t'+\epsilon)$ in which $0 \leq D^2 u(x, t)\leq M I_n$ with $u_{11}(x', t')=M$ at some interior point $(x', t')$ with respect to coordinates $x_1,...,x_n$ given by \eqref{rotatedcoord}.  It follows from \eqref{preserved1} and the strong maximum principle (Theorem 1, p.34, \cite{F}) that $u_{11}=M$ in $U_2\times[t'-\epsilon, t']$ and thus $v_{11}=1$ in $U_1\times[t'-\epsilon, t']$.

 Now for any $(r'', t'')\in\R^n\times[0, t']$ and let $\gamma(s)$, $s\in[0, 1]$, be a line segment in space-time such that $\gamma(0)=(r',t')$ and $\gamma(1)=(r'',t'')$.  Let $A$ be the set of $\bar{s} \in[0, 1]$ for which $v_{11}(\gamma(s))=1$ for all $s\in[0, \bar{s}]$.  Then the above argument shows that $A$ is in fact open and non-empty.  Moreover, $A$ is clearly closed by continuity and we then conclude that $A=[0, 1]$ and in particular,  $v_{11}(r'', t'')=1$.  This established the claim.

 {\bf Claim 2:} In Claim 1, we in fact have $v(r, t)=\frac{r_1^2}{2}+w(r_{2},...,r_n,t)$ on $\R^n\times[0, \infty)$.
 
  Integrating $v_{11}$ twice with respect to $r_1$ gives $$v(r, t)=\frac{r_1^2}{2}+r_1w_1(r_{2},...,r_n,t) + w_2(r_{2},...,r_n,t)$$ on $\R^n\times[0, t']$ for some functions $w_1$ and $w_2$.  It follows that $w_1$ must in fact be linear with respect to $x_2,...,x_n$ as otherwise $D^2 v$ would be unbounded on $\R^n\times[0, t')$ thus contradicting our assumption on that $-I_n<D^2 v(r, t)\leq I_n$ for all $r$ and $t$.  Our assumption that $D^2v(r',t')$ is diagonal then implies that $w_1$ must in fact be constant in space.  Finally, as the right hand side of \eqref{PMA} is uniformly bounded in absolute value from which we further conclude that is in fact constant in time as well and thus after a possible translation of the coordinate $r_1$ we have $$v(r, t)=\frac{r_1^2}{2}+w_3(r_{2},...,r_n,t)$$      on $\R^n\times[0, t']$ for some function $w_3$.  Now observe that up to the addition of a time dependent constant, $w_3(r_{2},...,r_n,t)$ solves \eqref{PMA} on $\R^{n-1}\times[0, t']$ and by Theorem \ref{entire} this extends to a smooth longtime solution which we still denote as $w_3(r_{2},...,r_n,t)$.  In particular $\frac{r_1^2}{2}+w_3(r_{2},...,r_n,t)$ is also a longtime 
solution to  \eqref{PMA} and  it follows from the uniqueness result in \cite{chen-pang} that the above representation of $v(r, t)$ holds on $\R^n\times[0, \infty)$.

 The lemma follows by iterating the arguments above starting with the function $w$ in Claim 2.\end{proof}

\begin{proof}[Proof of Theorem \ref{entire3}]
 Now let $u_0$ be a $C^{1,1}$ locally weakly convex function as in Theorem \ref{entire3}.  Using $\sigma=\pi/4$ in \eqref{rotatedcoord} to change coordinates on $\C^n$ and noting \eqref{rotatedlambda's} (see also \cite{Y}), we represent the Lagrangian graph $L=\{(x, u_0(x)) | x\in \R^n\}$ in the coordinates $z^j$ as $L=\{(r, v_0(r)) | r\in \R^n\}$ in the coordinates $w^j$ where $v_0$ satisfies $-I_n\leq D^2 v_0< I_n$.  Let $v(r, t)$ be the long time solution to \eqref{PMA} with initial condition $v_0$ given by Theorem \ref{entire}.
Then from Lemma 2.2 and Lemma  \ref{l1} we have $-I_n\leq D^2 v(r,t)< I_n$ for all $r$, $t\geq 0$.  Moreover, applying Lemma \ref{l2} to $-v(r, t)$ we further conclude that either $-I_n< D^2 v(r,t)< I_n$ $r$, $t> 0$ or \begin{equation}\label{splitting}v(r, t)=-\frac{r_1^2}{2}+\cdot\cdot\cdot- \frac{r_k^2}{2}+w(r_{k+1},...,r_n,t)\end{equation} on $\R^n\times[0, T)$ where $k>0$ and $-I_n<D^2 w(r, t)< I_n$ for all $r\in\R^n$, $t\geq 0$.  Let $L_t=\{(r, v(r, t)) | (r, t)\in \R^n\times[0, \infty)\}$  be the corresponding family of Lagrangian graphs in $\C^n$.  Then by \eqref{rotatedlambda's}, $L_t$ will correspond to a family of Lagrangian graphs $\{(x, Du(x, t)) | (x, t)\in \R^n\times[0, \infty)\}$ such that $u(x, t)$ is a longtime solution to \eqref{PMA} satisfying (i) in Theorem \ref{entire3}.  Now note that as $v(x, t)$ satisfies (ii) and (iii) in Theorem \ref{entire} it also satisfies (ii) and (iii) in Theorem \ref{entire3}.  It follows that $u(x, t)$ must then also satisfy (ii) and (iii) in Theorem \ref{entire3}.  The uniqueness of $u(x,t)$ follows from the uniqueness result in \cite{chen-pang}.  \end{proof}

 \begin{proof}[Proof of Theorem \ref{entire4}]
 Let $u_0$ be a locally $C^{1,1}$ function satisfying \eqref{supercritical}.  Then $u_0$ is automatically convex and by Theorem \ref{entire3} there exists a longtime convex solution  $u(x, t)$ to \eqref{PMA} with initial condition $u_0$.  In particular, note that $u(x,t)$ satisfies (ii) and (iii) in Theorem \ref{entire3}.   It will be convenient here to define the operator $$\Theta(A):=\displaystyle\sum_{i=1}^{n} \arctan \lambda_i (A)$$ on symmetric real $n\times n$ matrices $A$ where the $\lambda_i$'s are the eigenvalues of $A$.   A direct computation shows that as $u(x, t)$ solves \eqref{PMA}, $\Theta(D^2 u(x, t))$ evolves according to 
\begin{equation}\label{thetaevolve}
\partial_t  \Theta = \displaystyle \sum_{i,j=1}^{n}g^{ij} \partial^2 _{ij} \Theta.
\end{equation}
We would like to use \eqref{thetaevolve} and the maximum principle (Theorem 1, p.34, \cite{F}) to conclude that \eqref{supercritical} is thus preserved for all $t>0$.  One difficulty here is that $\Theta(D^2 u(x, t))$ is not necessarily continuous at $t=0$. Another difficulty is that   $D^2 u(x, t)$ is not neccesarily bounded above, and thus the symbol $g^{ij}$  is not necessarily bounded below (by a positive constant) on $\R^n$ for $t>0$.  To overcome this we will need to transform and approximate our solution $u(x,t)$ through the following sequence of steps.

{\bf Step 1 (small rotation):} We begin using \eqref{rotatedcoord}, with $\sigma=\sigma_0\in(0, \pi/2)$ to be chosen in a moment, to change coordinates on $\C^n$ and represent the Lagrangian graphs $L_t=\{(x, u(x,t)) | x\in \R^n\}$ in the coordinates $z^j$ as $L_t=\{(r, v(r,t)) | r\in \R^n\}$ in the coordinates $w^j$ for some family $v(r,t)$ with $(r,t)\in \R^n\times[0, \infty)$.  By \eqref{rotatedlambda's} we have \begin{equation}\label{sigmacritical}\Theta (D^2 v(r,0))  \geq (n-1)\frac{\pi}{2}-n\sigma_0\end{equation} and by \eqref{rotatedlambda's} and the convexity of $u(x,t)$ we have \begin{equation}\label{123}-K(\sigma_0) \leq D^2 v(r,t) \leq 1/K(\sigma_0)\end{equation} for all $r,t$ where $K(\sigma_0)\to 0$ as $\sigma_0 \to 0$.   
\vspace{12pt}

{\bf Step 2 (approximation):} Let $v^k_0$ be the sequence of approximations of $v_0=v(r,0)$ constructed in Lemma 2.1.  Then for each $k$ we have  $-K(\sigma_0) \leq D^2 v^k_0 \leq 1/K(\sigma_0)$ by \eqref{123}.  Moreover, $\sup_{r\in\R^n}|D^l v^k_0(r)| <\infty$ for all $l \geq 3$.   Now we show that   \begin{equation}\label{12345}\Theta (D^2 v^k_0)  \geq (n-1)\frac{\pi}{2}-n\sigma_0\end{equation} is satisfied for all $k$.

   Fix $r \in \R^n$ and $k$.  By \eqref{approxx} we have $$D^2 v^k_0(r)=\int_{\R^n} D^2 v_0(y)K\left(r, y, \frac{1}{k}\right)dy.$$ Approximating by the Riemann sums, we can find a double sequence $\{p_{ij}\} \subset \R^n$  and a sequence $\{j_i\} \subset {\mathbb Z}^+$ for which  \begin{equation}\label{approxint}D^2 v^k_0 (r)= \lim_{i\to\infty} \sum_{j=1}^{j_i} î  D^2 v_0(p_{ij})K\left(r, p_{ij}, \frac{1}{k}\right) \frac{1}{i^n}.\end{equation}  On the other hand, $$\int_{\R^n} K\left(r, y, \frac{1}{k}\right)dy=1$$ and we may then further assume $$B_i:=\sum_{j=1}^{j_i} K\left(r, p_{ij}, \frac{1}{k}\right) \frac{1}{i^n}\to 1$$ as $i\to \infty$.   By  \eqref{approxint} we then have 
\begin{equation}\label{approxint2}D^2 v^k_0 (r)= \lim_{i\to\infty} \sum_{j=1}^{j_i} î  D^2 v_0(p_{ij})A_{ij} \end{equation}
where $A_{ij}= K\left(r, p_{ij}, \frac{1}{k}\right) / (i^n B_i)$ and in particular $\sum_{j=1}^{j_i} A_{ij} =1$ while $A_{ij} \geq 0$ for all $i,j$.  Now since $(n-1)\frac{\pi}{2}-n\sigma_0 > (n-2)\frac{\pi}{2}$ by our choice of $\sigma_0$, the results in \cite{Y2} assert that the set of symmetric $n\times n$ matrices $A$ for which $\Theta \geq  (n-1)\frac{\pi}{2}-n\sigma_0$ is a convex set $S$ in the space of real $n\times n$ symmetric matrices.  This, \eqref{approxint2} and the fact that $D^2 v_0(p_{ij}) \in S$ for all $i, j$ imply $D^2 v^k_0 (r)  \in S$.  Thus \eqref{12345} holds for each $k$.

{\bf Step 3 ($\pi/4$ rotation):} Now we use \eqref{rotatedcoord} as in Step 1, but with $\sigma=\pi/4$, to obtain from $v(r, t)$ and the $v^k_0(r)$'s  a corresponding family $w(p,t)$ and sequence  $w^k_0(p)$.  In particular, $w(p,t)$ is a longtime solution to \eqref{PMA} and  the $w^k_0$'s will satisfy \eqref{hesscond} in Theorem \ref{entire}, provided $\sigma_0>0$ is chosen sufficiently small and we  will assume such a choice of $\sigma_0$ has been made.  They will also satisfy \begin{equation}\label{123456}\Theta (D^2 w^k_0(p))  \geq (n-1)\frac{\pi}{2}-n\sigma_0-n\frac{\pi}{4}\end{equation} by  \eqref{rotatedlambda's}. Thus for each $k$,  Theorem \ref{entire} gives a longtime solution $w^k(p,t)$ to \eqref{PMA} with initial condition $w^k_0$ satisfying  $\sup_{r\in\R^n}|D^l w^k(p,t)| <\infty$ for all $l \geq 2$ and $t\geq 0$.  It follows from \eqref{123456}, \eqref{thetaevolve}, \eqref{rotatedlambda's} and the weak maximum principle (Theorem 9, p.43, \cite{F}) that \begin{equation}\label{1234}\Theta (D^2 w^k(p,t))  \geq (n-1)\frac{\pi}{2}-n\sigma_0-n\frac{\pi}{4}\end{equation} for all $(p,t)$.  Now using Theorem \ref{entire} and arguing as in  the beginning of the proof of Theorem \ref{entire}, we see that some subsequence of the $w^k(p,t)$'s converge smoothly and uniformly on compact subsets of $\R^n\times(0, \infty)$ to a smooth limit solution to \eqref{PMA} on  $\R^n\times(0, \infty)$.  By the uniqueness result in \cite{chen-pang} and the definition of $w^k_0$, we see this limit solution is in fact the solution $w(p,t)$.  In particular, $w(p,t)$ must satisfy \eqref{1234} for all $(p,t)$. 


Rotating back to the original coordinates, we conclude from the last statement above that $u(x,t)$ must satisfy \eqref{supercritical} for all $t\geq 0$.  Thus either \eqref{supercritical} holds with strict inequality for all $t>0$ or there exists some $(x',t')\in\R^n\times(0, \infty)$ at which equality holds in \eqref{supercritical} in which case \eqref{thetaevolve} and the strong maximum principle (Theorem 1, p.34, \cite{F}) give
\begin{equation}\label{sigmacritical2}\Theta (D^2 u(x,t)) =(n-1)\frac{\pi}{2}\end{equation}in $\R^n\times (0, t']$.
In this case, integrating \eqref{PMA} in $t$ and noting the continuity of $u(x,t)$ in $t$ (for all $t\geq0$) we obtain $$u(x, t)=u(x,t')+(n-1)\frac{\pi}{2}(t-t')$$ for all $t\in[0, t']$, and thus for all $t\in[0, \infty)$ by the uniqueness result in \cite{chen-pang}.  In particular, $D^2 u(x,t)$ satisfies \eqref{sigmacritical2} for all $t\geq 0$.  On the other hand, $u(x,t')$ is smooth in $x$ and it follows that $u_0(x)=u(x,0)$ is a smooth convex solution to the special Lagrangian equation $\Theta (D^2 u_0(x))=  (n-1)\frac{\pi}{2}$ on $\R^n$ and is thus quadratic by the Bernstein theorem in \cite{Y}.  This concludes the proof of Theorem \ref{entire4}.\end{proof}

\bibliographystyle{amsplain}

\end{document}